\documentclass[fleqn]{book}
\usepackage{latexsym,amssymb,graphicx,makeidx,kcp}

\newcommand{\be}{\begin{enumerate}}
\newcommand{\ee}{\end{enumerate}}
\newcommand{\bi}{\begin{itemize}}
\newcommand{\ei}{\end{itemize}}
\newcommand{\bc}{\begin{center}}
\newcommand{\ec}{\end{center}}
\newcommand{\bsp}{\begin{sloppypar}}
\newcommand{\esp}{\end{sloppypar}}

\newcommand{\sglsp}{\ }
\newcommand{\dblsp}{\ \ }

\newtheorem{note}{Note}

\renewcommand{\phi}{\varphi}
\newcommand{\iotaAlt}{\mbox{\it \i}}
\newcommand{\iotaAltS}{\mbox{{\scriptsize \it \i}}}
\newcommand{\qzero}{${\cal Q}_0$}
\newcommand{\qzerou}{${\cal Q}^{\rm u}_{0}$}
\newcommand{\qzerouplus}{$\overline{{\cal Q}^{\rm u}_{0}}$}
\newcommand{\sC}{\mbox{$\cal C$}}
\newcommand{\sD}{\mbox{$\cal D$}}
\newcommand{\sG}{\mbox{$\cal G$}}
\newcommand{\sH}{\mbox{$\cal H$}}
\newcommand{\sJ}{\mbox{$\cal J$}}
\newcommand{\sL}{\mbox{$\cal L$}}
\newcommand{\sM}{\mbox{$\cal M$}}
\newcommand{\sP}{\mbox{$\cal P$}}
\newcommand{\sS}{\mbox{$\cal S$}}
\newcommand{\sT}{\mbox{$\cal T$}}
\newcommand{\sV}{\mbox{$\cal V$}}
\newcommand{\seq}[1]{{\langle #1 \rangle}}
\newcommand{\set}[1]{{\{ #1 \}}}
\newcommand{\mname}[1]{\mbox{\sf #1}}

\newcommand{\NegAlt}{{\sim}}
\newcommand{\And}{\wedge}
\newcommand{\Implies}{\supset}
\newcommand{\Or}{\vee}

\newcommand{\Forall}{\forall}

\newcommand{\Forsome}{\exists}

\newcommand{\Iota}{\mbox{\rm I}}
\newcommand{\IsDef}{\downarrow}
\newcommand{\IsUndef}{\uparrow}

\newcommand{\QuasiEqual}{\simeq}
\newcommand{\Undefined}{\bot}

\newcommand{\IsDefApp}{\!\IsDef}
\newcommand{\IsUndefApp}{\!\IsUndef}
\newcommand{\invertediota}{\rotatebox[origin=c]{180}{$\iota$}}
\newcommand{\pf}{\mbox{\sc pf}}
\newcommand{\pfstar}{$\pf^{\ast}$}
\newcommand{\imps}{\mbox{\sc imps}}
\newcommand{\tps}{\mbox{\sc tps}}
\newcommand{\hol}{\mbox{\sc hol}}
\newcommand{\pvs}{\mbox{\sc pvs}}
\newcommand{\lutins}{\mbox{\sc lutins}}
\newcommand{\sttwu}{\mbox{{\sc stt}{\small w}{\sc u}}}
\newcommand{\proves}[2]{#1 \vdash #2}
\newcommand{\wff}[1]{{\rm wff}_{#1}}
\newcommand{\cwff}[1]{{\rm cwff}_{#1}}
\newcommand{\wffs}[1]{{\rm wffs}_{#1}}
\newcommand{\cwffs}[1]{{\rm cwffs}_{#1}}

\collection
\begin{document}

\paper[Andrews' Type Theory with Undefinedness] 
{Andrews' Type Theory
with\\ Undefinedness\thanks{Originally published in: C. Benzm\"uller, C. Brown,
J. Siekmann, and R. Statman, eds., \emph{Reasoning in Simple Type Theory: 
Festschrift in Honor of Peter B. Andrews on his 70th Birthday}, 
\emph{Studies in Logic}, pp.~223--242, College Publications, 2008.}}
{William M. Farmer
\\[1ex] {\normalsize Revised: June 28, 2014}
}

\addtocounter{footnote}{+1}

\begin{abstract}
{\qzero} is an elegant version of Church's type theory formulated and
extensively studied by Peter~B.~Andrews.  Like other traditional
logics, {\qzero} does not admit undefined terms.  The
\emph{traditional approach to undefinedness} in mathematical practice
is to treat undefined terms as legitimate, nondenoting terms that can
be components of meaningful statements.  {\qzerou} is a modification
of Andrews' type theory {\qzero} that directly formalizes the
traditional approach to undefinedness.  This paper presents {\qzerou}
and proves that the proof system of {\qzerou} is sound and complete
with respect to its semantics which is based on Henkin-style general
models.  The paper's development of {\qzerou} closely follows Andrews'
development of {\qzero} to clearly delineate the differences between
the two systems.
\end{abstract}

\section{Introduction}

In 1940 Alonzo Church introduced in~\cite{Church40} a version of
simple type theory with lambda-notation now known as \emph{Church's
type theory}.  Church's students Leon Henkin and Peter B. Andrews
extensively studied and refined Church's type theory.  Henkin proved
that Church's type theory is complete with respect to a semantics
based on general models~\cite{Henkin50} and showed that Church's type
theory can be reformulated so that it is based on only the primitive
notions of function application, function abstraction, equality, and
(definite) description~\cite{Henkin63}.  Andrews devised a simple and
elegant proof system for Henkin's reformulation of Church's type
theory~\cite{Andrews63}.  He also formulated a version of Church's
type theory called {\qzero} that employs the ideas developed by
Church, Henkin, and himself.  {\qzero} is meticulously described and
analyzed in~\cite{Andrews02} and is the logic of the {\tps} Theorem
Proving System~\cite{AndrewsEtAl96}.

Church's type theory has had a profound impact on many areas of
computer science, especially programming languages, automated
reasoning, formal methods, type theory, and formalized mathematics.
It is the fountainhead of a long stream of typed lambda calculi that
includes systems such as \emph{System F}~\cite{GirardEtAl89},
\emph{Martin-L\"of type theory}~\cite{Martin-Lof84}, and the
\emph{Calculus of Constructions}~\cite{CoquandHuet88}.  Several
computer theorem proving systems are based on versions of Church's
type theory including {\hol}~\cite{GordonMelham93},
{\imps}~\cite{FarmerEtAl93,FarmerEtAl96}, Isabelle~\cite{Paulson94},
ProofPower~\cite{Lemma1Ltd00alt}, {\pvs}~\cite{OwreEtAl96}, and
{\tps}.

One of the principal virtues of Church's type theory is that it has
great expressivity, both theoretical and practical. However, like
other traditional logics, Church's type theory assumes that terms are
always defined.  Despite the fact that undefined terms are commonplace
in mathematics (and computer science), undefined terms cannot be
directly expressed in Church's type theory---as they are in
mathematical practice.

A term is \emph{undefined} if it has no prescribed meaning or if it
denotes a value that does not exist.\footnote{Some of the text in this
section concerning undefinedness is taken from~\cite{Farmer04}.}
There are two main sources of undefinedness in mathematics.  The first
source is terms that denote an application of a function.  A function
$f$ usually has both a \emph{domain of definition} $D_f$ consisting of
the values at which it is defined and a \emph{domain of application}
$D^{\ast}_{f}$ consisting of the values to which it may be applied.
(The domain of definition of a function is usually called simply the
\emph{domain} of the function.)  These two domains are not always the
same, but obviously $D_f \subseteq D^{\ast}_{f}$.  A \emph{function
application} is a term $f(a)$ that denotes the application of a
function $f$ to an argument $a \in D^{\ast}_{f}$.  $f(a)$ is
\emph{undefined} if $a \not\in D_{f}$.  We will say that a function is
\emph{partial} if $D_f \not= D^{\ast}_{f}$ and \emph{total} if $D_f =
D^{\ast}_{f}$.

\bsp The second source of undefinedness is terms that are intended to
uniquely describe a value.  A \emph{definite description} is a term
$t$ of the form ``the $x$ that has property $P$''. $t$ is
\emph{undefined} if there is no unique $x$ (i.e., none or more than
one) that has property $P$.  Definite descriptions are quite common in
mathematics but often occur in a disguised form.  For example, ``the
limit of $\mname{sin}\frac{1}{x}$ as $x$ approaches 0'' is a definite
description---which is undefined since the limit does not exist. \esp

There is a \emph{traditional approach to undefinedness} that is widely
practiced in mathematics and even taught to some extent to students in
secondary school.  This approach treats undefined terms as legitimate,
nondenoting terms that can be components of meaningful statements.
The traditional approach is based on three principles:

\be

  \item Atomic terms (i.e., variables and constants) are always
  defined---they always denote something.

  \item Compound terms may be undefined.  A function application
  $f(a)$ is undefined if $f$ is undefined, $a$ is undefined, or $a
  \not\in D_f$.  A definite description ``the $x$ that has property
  $P$'' is undefined if there is no $x$ that has property $P$ or there
  is more than one $x$ that has property $P$.

  \item Formulas are always true or false, and hence, are always
  defined.  To ensure the definedness of formulas, a function
  application $p(a)$ formed by applying a predicate $p$ to an argument
  $a$ is \emph{false} if $p$ is undefined, $a$ is undefined, or $a
  \not\in D_p$.

\ee

A logic that formalizes the traditional approach to undefinedness has
two advantages over a traditional logic that does not.  First, the use
of the traditional approach in informal mathematics can be directly
formalized, yielding a result that is close to mathematical practice.
Second, statements involving partial functions and undefined terms can be
expressed very concisely.  In particular, assumptions about the
definedness of terms and functions often do not have to be made
explicit.  Concise informal mathematical statements involving partial
functions or undefinedness can usually only be expressed in a
traditional logic by verbose statements in which definedness
assumptions are explicit.  For evidence and further discussion of
these assertions, see~\cite{Farmer04}.

We presented in~\cite{Farmer90} a version of Church's type system
named {\pf} that formalizes the traditional approach to undefinedness.
{\pf} is the basis for {\lutins}~\cite{Farmer93b,Farmer94}, the logic
of the {\imps} theorem proving
system~\cite{FarmerEtAl93,FarmerEtAl96}.  The paper \cite{Farmer90}
includes a proof that {\pf} is complete with respect to a Henkin-style
general models semantics.  The proof, however, contains a mistake: the
tautology theorem does not hold in {\pf} as claimed.  This mistake can
be corrected by adding modus ponens and a technical axiom schema
involving equality to {\pf}'s proof system.  In~\cite{Farmer04} we
introduced a version of Church's type system with undefinedness called
{\sttwu} which is simpler than {\pf}.  The proof system of {\sttwu} is
claimed to be complete, but a proof of completeness is not given
in~\cite{Farmer04}.

The purpose of this paper is to carefully show what changes have to be
made to Church's type theory in order to formalize the traditional
approach to undefinedness.  We do this by presenting a modification of
Andrews' type theory {\qzero} called {\qzerou}.  Our goal is to keep
{\qzerou} as close to {\qzero} as possible, changing as few of the
definitions in~\cite{Andrews02} concerning {\qzero} as possible.  We
present the syntax, semantics and proof system of {\qzerou} and prove
that the proof system is sound and complete with respect to its
semantics.  A series of notes indicates precisely where and how
{\qzero} and {\qzerou} diverge from each other.

Our presentation of {\qzerou} differs from the presentation of {\pf}
in~\cite{Farmer90} in the following ways:

\be

  \item The notation and terminology for {\qzerou} is almost identical
  to the notation and terminology for {\qzero} given
  in~\cite{Andrews02} unlike the notation and terminology for {\pf}.

  \item The semantics of {\qzerou} is simpler than the semantics of
  {\pf}.

  \item The proof system of {\qzerou} is complete unlike the proof
  system of {\pf}.

  \item The proof of the completeness theorem for {\qzerou} is
  presented in greater detail than the (erroneous) proof of the
  completeness theorem for {\pf}.

\ee

The paper is organized as follows.  The syntax of {\qzerou} is defined
in section~\ref{sec:syntax}.  A Henkin-style general models semantics
for {\qzerou} is presented in section~\ref{sec:semantics}.
Section~\ref{sec:definitions} introduces several important defined
logical constants and notational abbreviations.
Section~\ref{sec:proofsystem} gives the proof system of {\qzerou}.
Some metatheorems of {\qzerou} and the soundness and completeness
theorems for {\qzerou} are proved in sections~\ref{sec:metatheorems}
and \ref{sec:completeness}, respectively.  The paper ends with a
conclusion in section~\ref{sec:conclusion}.

The great majority of the definitions for {\qzerou} are exactly the
same as those for {\qzero} given in~\cite{Andrews02}.  We repeat
only the most important and least obvious definitions for {\qzero};
for the others the reader is referred to~\cite{Andrews02}.

\section{Syntax of {\qzerou}} \label{sec:syntax}

The syntax of {\qzerou} is almost exactly the same as that of
{\qzero}.  The only difference is that just one iota constant is
primitive in {\qzero}, while infinitely many iota constants are
primitive in {\qzerou}.

A \emph{type symbol} of {\qzerou} is defined inductively as follows:

\be

  \item {\iotaAlt} is a type symbol.

  \item o is a type symbol.

  \item If $\alpha$ and $\beta$ are type symbols, then $(\alpha\beta)$
  is a type symbol.

\ee
Let $\sT$ denote the set of type symbols.  $\alpha,\beta,\gamma,
\ldots$ are syntactic variables ranging over type symbols.  When there
is no loss of meaning, matching pairs of parentheses in type symbols
may be omitted.  We assume that type combination associates to the
left so that a type of the form $((\alpha\beta)\gamma)$ may be written
as $\alpha\beta\gamma$.

The \emph{primitive symbols} of {\qzerou} are the following:

\be

  \item \emph{Improper symbols}: [, ], $\lambda$.

  \item A denumerable set of \emph{variables} of type $\alpha$ for
  each $\alpha \in \sT$: $f_\alpha$, $g_\alpha$, $h_\alpha$,
  $x_\alpha$, $y_\alpha$, $z_\alpha$, $f^{1}_{\alpha}$,
  $g^{1}_{\alpha}$, $h^{1}_{\alpha}$, $x^{1}_{\alpha}$,
  $y^{1}_{\alpha}$, $z^{1}_{\alpha}$, $\dots$\ .

  \item \emph{Logical constants}: $\mname{Q}_{((o\alpha)\alpha)}$ for
  each $\alpha \in \sT$ and $\iota_{(\alpha(o\alpha))}$ for each
  $\alpha \in \sT$ with $\alpha \not= o$.

  \item An unspecified set of \emph{nonlogical constants} of various
  types.

\ee
$\textbf{x}_\alpha, \textbf{y}_\alpha, \textbf{z}_\alpha,
\textbf{f}_\alpha, \textbf{g}_\alpha, \textbf{h}_\alpha, \ldots$ are
syntactic variables ranging over variables of type $\alpha$.

\medskip
\begin{note}[Iota Constants]
Only $\iota_{\iotaAltS(o\iotaAltS)}$ is a primitive logical constant
in {\qzero}; each other $\iota_{\alpha(o\alpha)}$ is a nonprimitive
logical constant in {\qzero} defined according to an inductive scheme
presented by Church in~\cite{Church40}
(see~\cite[pp.~233--4]{Andrews02}).  We will see in the next section
that the iota constants have a different semantics in {\qzerou} than
in {\qzero}.  As a result, it is not possible to define the iota
constants in {\qzerou} as they are defined in {\qzero}, and thus they
must be primitive in {\qzerou}.  Notice that $\iota_{o(oo)}$ is not a
primitive logical constant of {\qzerou}.  It has been left out because
it serves no useful purpose.  It can be defined as a nonprimitive
logical constant as in~\cite[p.~233]{Andrews02} if
desired.\hfill$\blacksquare$
\end{note}
\medskip

We are now ready to define a \emph{wff of type $\alpha$}
($\textit{wff}_\alpha$).  $\textbf{A}_\alpha, \textbf{B}_\alpha,
\textbf{C}_\alpha, \ldots$ are syntactic variables ranging over wffs
of type $\alpha$.  A $\textit{wff}_\alpha$ is then defined inductively
as follows:

\be

  \item A variable or primitive constant of type $\alpha$ is a
    $\wff{\alpha}$.

  \item $[\textbf{A}_{\alpha\beta}\textbf{B}_\beta]$ is a
  $\wff{\alpha}$.

  \item $[\lambda\textbf{x}_\beta\textbf{A}_\alpha]$ is a
  $\wff{\alpha\beta}$.

\ee
A wff of the form $[\textbf{A}_{\alpha\beta}\textbf{B}_\beta]$ is
called a \emph{function application} and a wff of the form
$[\lambda\textbf{x}_\beta\textbf{A}_\alpha]$ is called a
\emph{function abstraction}.  When there is no loss of meaning,
matching pairs of square brackets in wffs may be omitted.  We assume
that wff combination of the form
$[\textbf{A}_{\alpha\beta}\textbf{B}_\beta]$ associates to the left so
that a wff
$[[\textbf{C}_{\gamma\beta\alpha}\textbf{A}_\alpha]\textbf{B}_\beta]$
may be written as
$\textbf{C}_{\gamma\beta\alpha}\textbf{A}_\alpha\textbf{B}_\beta$.

\section{Semantics of {\qzerou}} \label{sec:semantics}

The traditional approach to definedness is formalized in {\qzerou} by
modifying the semantics of {\qzero}.  Two principal changes are made
to the {\qzero} semantics: (1) The notion of a general model is
redefined to include partial functions as well as total functions.
(2) The valuation function for wffs is made into a partial function
that assigns a value to a wff iff the wff is defined according to the
traditional approach.

A \emph{frame} is a collection $\set{\sD_\alpha \;|\; \alpha \in \sT}$
of nonempty domains such that:

\be

  \item $\sD_o = \set{\mname{T},\mname{F}}$.

  \item For $\alpha, \beta \in \sT$, $\sD_{\alpha\beta}$ is some set
    of \emph{total} functions from $\sD_\beta$ to $\sD_\alpha$ if
    $\alpha = o$ and is some set of \emph{partial and total} functions
    from $\sD_\beta$ to $\sD_\alpha$ if $\alpha \not= o$.\footnote{The
      condition that a domain $D_{o\beta}$ contains only total
      functions is needed to ensure that the law of extensionality
      holds for predicates.  This condition is weaker than the
      condition used in the semantics for {\pf}~\cite{Farmer90} and
      its extended versions {\pfstar}~\cite{Farmer93b} and
      {\lutins}\cite{Farmer93b,Farmer94}.  In these logics, a domain
      $D_{\gamma}$ contains only total functions iff $\gamma$ has the
      form $o\beta_1 \cdots \beta_n$ where $n \ge 1$.  The weaker
      condition, which is due to Aaron~Stump~\cite{Stump03}, yields a
      semantics that is simpler.}

\ee 
$\sD_o$ is the \emph{domain of truth values}, $\sD_{\iotaAltS}$ is the
\emph{domain of individuals}, and for $\alpha, \beta \in \sT$,
$\sD_{\alpha\beta}$ is a \emph{function domain}.  For all $\alpha \in
\sT$, the \emph{identity relation} on $\sD_\alpha$ is the total
function $q \in \sD_{o\alpha\alpha}$ such that, for all $x,y \in
\sD_\alpha$, $q(x)(y) = \mname{T}$ iff $x = y$.  For all $\alpha \in
\sT$ with $\alpha \not= o$, the \emph{unique member selector} on
$\sD_\alpha$ is the partial function $f \in \sD_{\alpha(o\alpha)}$
such that, for all $s \in \sD_{o\alpha}$, if the predicate $s$
represents a singleton $\set{x} \subseteq \sD_\alpha$, then $f(s) =
x$, and otherwise $f(s)$ is undefined.

\medskip
\begin{note}[Function Domains]
In a {\qzero} frame a function domain $\sD_{\alpha\beta}$ contains
only total functions, while in a {\qzerou} frame a function domain
$\sD_{o\beta}$ contains only total functions but a function domain
$\sD_{\alpha\beta}$ with $\alpha \not= o$ contains partial functions
as well as total functions.\hfill$\blacksquare$
\end{note}
\medskip

An \emph{interpretation} $\seq{\set{\sD_\alpha \;|\; \alpha \in \sT},
  \sJ}$ of {\qzerou} consists of a frame and a function $\sJ$ that
maps each primitive constant of {\qzerou} of type $\alpha$ to an
element of $\sD_\alpha$ such that $\sJ(\mname{Q}_{o\alpha\alpha})$ is
the identity relation on $\sD_\alpha$ for each $\alpha \in \sT$ and
$\sJ(\iota_{\alpha(o\alpha)})$ is the unique member selector on
$\sD_\alpha$ for each $\alpha \in \sT$ with $\alpha \not= o$.

\medskip
\begin{note}[Definite Description Operators]
The $\iota_{\alpha(o\alpha)}$ in {\qzero} are \emph{description
operators}: if $\textbf{A}_{o\alpha}$ denotes a singleton, then the
value of $\iota_{\alpha(o\alpha)}\textbf{A}_{o\alpha}$ is the unique
member of the singleton, and otherwise the value of
$\iota_{\alpha(o\alpha)}\textbf{A}_{o\alpha}$ is \emph{unspecified}.
In contrast, the $\iota_{\alpha(o\alpha)}$ in {\qzerou} are
\emph{definite description operators}: if $\textbf{A}_{o\alpha}$
denotes a singleton, then the value of
$\iota_{\alpha(o\alpha)}\textbf{A}_{o\alpha}$ is the unique member of
the singleton, and otherwise the value of
$\iota_{\alpha(o\alpha)}\textbf{A}_{o\alpha}$ is
\emph{undefined}.\phantom{XXXX}\hfill$\blacksquare$
\end{note}
\medskip

An \emph{assignment} into a frame $\set{\sD_\alpha \;|\; \alpha \in
\sT}$ is a function $\phi$ whose domain is the set of variables of
{\qzerou} such that, for each variable $\textbf{x}_\alpha$,
$\phi(\textbf{x}_\alpha) \in \sD_\alpha$.  Given an assignment $\phi$,
a variable $\textbf{x}_\alpha$, and $d \in \sD_\alpha$, let $(\phi :
\textbf{x}_\alpha/d)$ be the assignment $\psi$ such that
$\psi(\textbf{x}_\alpha) = d$ and $\psi(\textbf{y}_\beta) =
\phi(\textbf{y}_\beta)$ for all variables $\textbf{y}_\beta \not=
\textbf{x}_\alpha$.

An interpretation $\sM = \seq{\set{\sD_\alpha \;|\; \alpha \in \sT},
  \sJ}$ is a \emph{general model} for {\qzerou} if there is a binary
function $\sV^{\cal M}$ such that, for each assignment $\phi$ and wff
$\textbf{C}_\gamma$, either $\sV^{\cal M}_{\phi}(\textbf{C}_\gamma)
\in \sD_\gamma$ or $\sV^{\cal M}_{\phi}(\textbf{C}_\gamma)$ is
undefined and the following conditions are satisfied for all
assignments $\phi$ and all wffs $\textbf{C}_\gamma$:

\bi

  \item[(a)] Let $\textbf{C}_\gamma$ be a variable of {\qzerou}.  Then
  $\sV^{\cal M}_{\phi}(\textbf{C}_\gamma) = \phi(\textbf{C}_\gamma)$.

  \item[(b)] Let $\textbf{C}_\gamma$ be a primitive constant of
    {\qzerou}.  Then $\sV^{\cal M}_{\phi}(\textbf{C}_\gamma) =
    \sJ(\textbf{C}_\gamma)$.

  \item[(c)] Let $\textbf{C}_\gamma$ be $[\textbf{A}_{\alpha\beta}
  \textbf{B}_\beta]$.  If $\sV^{\cal
  M}_{\phi}(\textbf{A}_{\alpha\beta})$ is defined, $\sV^{\cal
  M}_{\phi}(\textbf{B}_\beta)$ is defined, and the function $\sV^{\cal
  M}_{\phi}(\textbf{A}_{\alpha\beta})$ is defined at the argument
  $\sV^{\cal M}_{\phi}(\textbf{B}_\beta)$, then
  \[\sV^{\cal M}_{\phi}(\textbf{C}_\gamma) = \sV^{\cal
  M}_{\phi}(\textbf{A}_{\alpha\beta})(\sV^{\cal
  M}_{\phi}(\textbf{B}_\beta)),\] the value of the function $\sV^{\cal
  M}_{\phi}(\textbf{A}_{\alpha\beta})$ at the argument $\sV^{\cal
  M}_{\phi}(\textbf{B}_\beta)$.  Otherwise, $\sV^{\cal
  M}_{\phi}(\textbf{C}_\gamma) = \mname{F}$ if $\alpha = o$ and
  $\sV^{\cal M}_{\phi}(\textbf{C}_\gamma)$ is undefined if $\alpha
  \not= o$.

  \item[(d)] Let $\textbf{C}_\gamma$ be
  $[\lambda\textbf{x}_\beta\textbf{B}_\alpha]$.  Then $\sV^{\cal
  M}_{\phi}(\textbf{C}_\gamma)$ is the (partial or total) function $f
  \in \sD_{\alpha\beta}$ such that, for each $d \in
  \sD_\beta$, $f(d) = \sV^{\cal M}_{(\phi : {\bf
  x}_\beta/d)}(\textbf{B}_\alpha)$ if $\sV^{\cal M}_{(\phi : {\bf
  x}_\alpha/d)}(\textbf{B}_\alpha)$ is defined and $f(d)$ is undefined
  if $\sV^{\cal M}_{(\phi : {\bf x}_\beta/d)}(\textbf{B}_\alpha)$ is
  undefined.

\ei

\medskip
\begin{note}[Valuation Function]
In {\qzero}, if $\sM$ is a general model, then $\sV^{\cal M}$ is total
and the value of $\sV^{\cal M}$ on a function abstraction is always a
total function.  In {\qzerou}, if $\sM$ is a general model, then
$\sV^{\cal M}$ is partial and the value of $\sV^{\cal M}$ on a
function abstraction can be either a partial or a total
function.\phantom{XXXX}\hfill$\blacksquare$
\end{note}
\medskip

\begin{proposition}
Let $\sM$ be a general model for {\qzerou}.  Then $\sV^{\cal M}$ is
defined on all variables, primitive constants, function abstractions,
and function applications of type $o$ and is defined on only a proper
subset of function applications of type $\alpha \not= o$.
\end{proposition}

\medskip
\begin{note}[Traditional Approach]
{\qzerou} clearly satisfies the three principles of the traditional
approach to undefinedness.  Like other traditional logics, {\qzero}
only satisfies the first principle.\hfill$\blacksquare$
\end{note}
\medskip

Let $\sH$ be a set of $\wffs{o}$ and $\sM$ be a general model for
{\qzerou}.  $\textbf{A}_o$ is \emph{valid} in $\sM$, written $\sM
\models \textbf{A}_o$, if $\sV^{\cal M}_{\phi}(\textbf{A}_o) =
\mname{T}$ for all assignments $\phi$.  $\sM$ is a \emph{general
model} for $\sH$ if $\sM \models \textbf{B}_o$ for all $\textbf{B}_o
\in \sH$.  $\textbf{A}_o$ is \emph{valid (in the general sense)} in
{\sH}, written $\sH \models \textbf{A}_o$, if $\sM \models
\textbf{A}_o$ for every general model $\sM$ for {\sH}.  $\textbf{A}_o$
is \emph{valid (in the general sense)} in {\qzerou}, written ${}
\models \textbf{A}_o$, if $\emptyset \models \textbf{A}_o$.

\medskip
\begin{note}[Mutual Interpretability] \label{note:mutual}
{\qzerou} can be interpreted in {\qzero} by viewing a function of type
$\alpha\beta$ in {\qzerou} as a function (predicate) of type
$o\alpha\beta$ in {\qzero}.  {\qzero} can be interpreted in {\qzerou}
by viewing a function of type $\alpha\beta$ in {\qzero} as a total
function of type $\alpha\beta$ in {\qzerou}.  Thus {\qzero} and
{\qzerou} are equivalent in the sense of being mutually
interpretable.\hfill$\blacksquare$
\end{note}

\section{Definitions and Abbreviations} \label{sec:definitions}

As Andrews does in~\cite[p.~212]{Andrews02}, we will introduce several
defined logical constants and notational abbreviations.  The former
includes constants for true and false, the propositional connectives,
and a canonical undefined wff.  The latter includes notation for
equality, the propositional connectives, universal and existential
quantification, defined and undefined wffs, quasi-equality, and
definite description.

\begin{tabbing}
\= 
\= $[\textbf{A}_\alpha = \textbf{B}_\alpha]$ \hspace{5mm}\= stands for \hspace{3mm}
\= $[\mname{Q}_{o\alpha\alpha}\textbf{A}_\alpha\textbf{B}_\alpha]$.\\

\> 
\> $T_o$ \> stands for
\> $[\mname{Q}_{ooo} = \mname{Q}_{ooo}]$.\\

\> 
\> $F_o$ \> stands for
\> $[\lambda x_o T_o] = [\lambda x_o x_o]$.\\

\> 
\> $[\Forall\textbf{x}_\alpha\textbf{A}_o]$ \> stands for
\> $[\lambda y_{\alpha} T_o] = [\lambda \textbf{x}_\alpha\textbf{A}_o]$.\\

\> 
\> $\wedge_{ooo}$ \> stands for
\> $[\lambda x_o \lambda y_o [[\lambda g_{ooo} [g_{ooo} T_o T_o]] = [\lambda g_{ooo} [g_{ooo} x_o y_o]]]]$.\\

\> 
\> $[\textbf{A}_o \And \textbf{B}_o]$ \> stands for
\> $[\wedge_{ooo}\textbf{A}_o\textbf{B}_o]$.\\

\> 
\> $\Implies_{ooo}$ \> stands for
\> $[\lambda x_o \lambda y_o [x_o = [x_o \And y_o]]]$.\\

\> 
\> $[\textbf{A}_o \Implies \textbf{B}_o]$ \> stands for
\> $[{\Implies_{ooo}}\textbf{A}_o\textbf{B}_o]$.\\

\> 
\> $\NegAlt_{oo}$ \> stands for
\> $[\mname{Q}_{ooo}F_o]$.\\

\> 
\> $[\NegAlt\textbf{A}_o]$ \> stands for
\> $[\NegAlt_{oo}\textbf{A}_o]$.\\

\> 
\> $\vee_{ooo}$ \> stands for
\> $[\lambda x_o \lambda y_o [\NegAlt[[\NegAlt x_o] \And [\NegAlt y_o]]]]$.\\

\> 
\> $[\textbf{A}_o \Or \textbf{B}_o]$ \> stands for
\> $[\vee_{ooo}\textbf{A}_o\textbf{B}_o]$.\\

\> 
\> $[\Forsome\textbf{x}_\alpha\textbf{A}_o]$ \> stands for
\> $[\NegAlt[\Forall\textbf{x}_\alpha\NegAlt\textbf{A}_o]]$.\\

\> 
\> $[\Forsome_1\textbf{x}_\alpha\textbf{A}_o]$ \> stands for
\> $[\Forsome y_\alpha
   [[\lambda\textbf{x}_\alpha\textbf{A}_o] = 
   \mname{Q}_{o\alpha\alpha}y_\alpha]]$\\
\> \> \> \> where $y_\alpha$ does not occur in $\textbf{A}_o$.\\

\> 
\> $[\textbf{A}_\alpha \not= \textbf{B}_\alpha]$ \> stands for 
\> $[\NegAlt[\textbf{A}_\alpha = \textbf{B}_\alpha]]$.\\

\> 
\> $[{\textbf{A}_\alpha\IsDefApp}]$ \> stands for
\> $[\Forsome x_\alpha [x_\alpha = \textbf{A}_\alpha]]$\\
\> \> \> \> where $x_\alpha$ does not occur in $\textbf{A}_\alpha$.\\

\> 
\> $[{\textbf{A}_\alpha\IsUndefApp}]$ \> stands for
\> $[\NegAlt[{\textbf{A}_\alpha\IsDefApp}]]$\\

\> 
\> $[\textbf{A}_\alpha \QuasiEqual \textbf{B}_\alpha]$ \> stands for 
\> $[{\textbf{A}_\alpha\IsDefApp} \Or {\textbf{B}_\alpha\IsDefApp}]
   \Implies [\textbf{A}_\alpha = \textbf{B}_\alpha]$.\\

\> 
\> $[\Iota\textbf{x}_\alpha\textbf{A}_o]$ \> stands for
\> $[\iota_{\alpha(o\alpha)}[\lambda\textbf{x}_\alpha\textbf{A}_o]]$.\\

\> 
\> $\Undefined_\alpha$ \> stands for
\> $[\Iota x_\alpha [x_\alpha \not= x_\alpha]]$ where $\alpha \not= o$.
\end{tabbing}

$[\Forsome_1\textbf{x}_\alpha\textbf{A}_o]$ asserts that there is a
unique $\textbf{x}_\alpha$ that satisfies $\textbf{A}_o$.

$[\Iota\textbf{x}_\alpha\textbf{A}_o]$ is called a \emph{definite
description}.  It denotes the unique $\textbf{x}_\alpha$ that
satisfies $\textbf{A}_o$.  If there is no or more than one such
$\textbf{x}_\alpha$, it is undefined.  Following Bertrand Russell and
Church, Andrews denotes this definite description operator as an
inverted lower case iota (\invertediota).  We represent this operator
by an (inverted) capital iota ($\Iota$).

$[{\textbf{A}_\alpha\IsDefApp}]$ says that $\textbf{A}_\alpha$ is
defined, and similarly, $[{\textbf{A}_\alpha\IsUndefApp}]$ says that
$\textbf{A}_\alpha$ is undefined.  $[\textbf{A}_\alpha \QuasiEqual
\textbf{B}_\alpha]$ says that $\textbf{A}_\alpha$ and
$\textbf{B}_\alpha$ are \emph{quasi-equal}, i.e., that
$\textbf{A}_\alpha$ and $\textbf{B}_\alpha$ are either both defined
and equal or both undefined.  $\Undefined_\alpha$ is a canonical
undefined wff of type $\alpha$.

\medskip
\begin{note}[Definedness Notation]
In {\qzero}, $[{\textbf{A}_\alpha\IsDefApp}]$ is always true,
$[{\textbf{A}_\alpha\IsUndefApp}]$ is always false, $[\textbf{A}_\alpha
\QuasiEqual \textbf{B}_\alpha]$ is always equal to $[\textbf{A}_\alpha
= \textbf{B}_\alpha]$, and $\Undefined_\alpha$ denotes an unspecified
value.\hfill$\blacksquare$
\end{note}
\medskip

\section{Proof System of {\qzerou}} \label{sec:proofsystem}

In this section we present the proof system of {\qzerou} which is
derived from the proof system of {\qzero}.  The issue of definedness
makes the proof system of {\qzerou} moderately more complicated than
the proof system for {\qzero}.  While {\qzero} has only five axiom
schemas and one rule of inference, {\qzerou} has the following
thirteen axiom schemas and two rules of inference:

\bi

  \item[] \textbf{A1 (Truth Values)}
  \[[g_{oo}T_o \And g_{oo}F_o] = \Forall x_o [g_{oo} x_o].\]

  \item[] \textbf{A2 (Leibniz' Law)} 
  \[[x_\alpha = y_\alpha] \Implies 
  [h_{o\alpha}x_\alpha = h_{o\alpha}y_\alpha].\]

  \item[] \textbf{A3 (Extensionality)}
  \[[f_{\alpha\beta} = g_{\alpha\beta}] = \Forall x_\beta 
  [f_{\alpha\beta}x_\beta \QuasiEqual g_{\alpha\beta}x_\beta].\]

  \item[] \textbf{A4 (Beta-Reduction)} 
  \[{\textbf{A}_\alpha \IsDefApp} \Implies 
  [[\lambda \textbf{x}_\alpha \textbf{B}_\beta]\textbf{A}_\alpha
    \QuasiEqual \mname{S}^{{\bf x}_\alpha}_{{\bf A}_\alpha}
    \textbf{B}_\beta]\] provided $\textbf{A}_\alpha$ is free for
  $\textbf{x}_\alpha$ in $\textbf{B}_\beta$.\footnote{$\mname{S}^{{\bf
        x}_\alpha}_{{\bf A}_\alpha}\textbf{B}_\beta$ is the result of
    substituting $\textbf{A}_\alpha$ for each free occurrence of
    $\textbf{x}_\alpha$ in $\textbf{B}_\beta$.}

  \medskip

  \item[] \textbf{A5 (Variables are Defined)} 
  \[{\textbf{x}_\alpha\IsDefApp}.\]

  \item[] \textbf{A6 (Primitive Constants are Defined)}
  \[{\textbf{c}_\alpha\IsDefApp} \dblsp 
  \mbox{where } \textbf{c}_\alpha \mbox{ is a primitive
    constant.}\footnotemark\]\footnotetext{Notice that
    $\textbf{c}_\alpha\IsDefApp$ is false if $\textbf{c}_\alpha$ is a
    defined constant $\Undefined_\alpha$ where $\alpha \not= o$.}

  \item[] \textbf{A7 (Function Abstractions are Defined)}
  \[{[\lambda \textbf{x}_\alpha \textbf{B}_\beta]\IsDefApp}.\]

  \item[] \textbf{A8 (Function Applications of Type $o$ are Defined)}
  \[{\textbf{A}_{o\beta}\textbf{B}_\beta \IsDefApp}.\]

  \item[] \textbf{A9 (Improper Function Application of Type $o$)}
  \[[{\textbf{A}_{o\beta}\IsUndefApp} \Or {\textbf{B}_\beta\IsUndefApp}]
  \Implies \NegAlt[\textbf{A}_{o\beta}\textbf{B}_\beta].\]

  \item[] \textbf{A10 (Improper Function Application of Type $\alpha \not= o$)}
  \[[{\textbf{A}_{\alpha\beta}\IsUndefApp} \Or {\textbf{B}_\beta\IsUndefApp}]
  \Implies {\textbf{A}_{\alpha\beta}\textbf{B}_\beta\IsUndefApp}
  \dblsp \mbox{where } \alpha \not= o.\]

  \item[] \textbf{A11 (Equality and Quasi-Quality)}
  \[{\textbf{A}_\alpha\IsDefApp} \Implies
  [{\textbf{B}_\alpha\IsDefApp} \Implies [[\textbf{A}_\alpha
  \QuasiEqual \textbf{B}_\alpha] \QuasiEqual [\textbf{A}_\alpha =
  \textbf{B}_\alpha]]].\]

  \item[] \textbf{A12 (Proper Definite Description)} 
  \[{\Forsome_1 \textbf{x}_\alpha \textbf{A}_o}
  \Implies [{[\Iota \textbf{x}_\alpha \textbf{A}_o]\IsDefApp} \And
  \mname{S}^{{\bf x}_\alpha}_{[{\rm I} {\bf x}_\alpha {\bf A}_o]}
  \textbf{A}_o] \dblsp \mbox{where }\alpha \not= o\] and provided $\Iota
  \textbf{x}_\alpha \textbf{A}_o$ is free for $\textbf{x}_\alpha$ in
  $\textbf{A}_o$.

  \medskip

  \item[] \textbf{A13 (Improper Definite Description)} 
  \[{\NegAlt[\Forsome_1 \textbf{x}_\alpha \textbf{A}_o]} \Implies 
  {[\Iota \textbf{x}_\alpha \textbf{A}_o]\IsUndefApp}
  \dblsp \mbox{where }\alpha \not= o.\]

  \item[] \textbf{R1 (Quasi-Equality Substitution)}\dblsp From
  $\textbf{A}_\alpha \QuasiEqual \textbf{B}_\alpha$ and $\textbf{C}_o$
  infer the result of replacing one occurrence of $\textbf{A}_\alpha$
  in $\textbf{C}_o$ by an occurrence of $\textbf{B}_\alpha$, provided
  that the occurrence of $\textbf{A}_\alpha$ in $\textbf{C}_o$ is not
  (an occurrence of a variable) immediately preceded by $\lambda$.

  \medskip

  \item[] \textbf{R2 (Modus Ponens)}\dblsp From $\textbf{A}_o$ and
  $\textbf{A}_o \Implies \textbf{B}_o$ infer $\textbf{B}_o$.

\ei

\medskip
\begin{note}[Axiom Schemas]
The axiom schemas A1, A2, A3, A4, and A12 of {\qzerou} correspond to
the five axiom schemas of {\qzero}.  A1 and A2 are exactly the same as
the first and second axiom schemas of {\qzero}.  A3 and A4 are
modifications of the third and fourth axiom schemas of {\qzero}.  A3
is the axiom of extensionality for partial and total functions, and A4
is beta-reduction for functions that may be partial and arguments that
may be undefined.

The seven axiom schemas A5--A11 of {\qzerou} deal with the definedness
of wffs.  A5 and A6 address the first principle of the traditional
approach to undefinedness, A10 addresses the second principle, and A8
and A9 address the third principle.  A7 states that a function
abstraction always denotes some function, either partial or total.
And A11 is a technical axiom schema for identifying equality with
quasi-equality when applied to defined wffs.

The last two axiom schemas of {\qzerou} state the
properties of definite descriptions.  A12 states that proper
definite descriptions are defined and denote the unique value
satisfying the description; it corresponds to the fifth axiom schema
of {\qzero}.  A13 states that improper definite descriptions
are undefined.  The proof system of {\qzero} leaves improper definite
descriptions unspecified.~$\blacksquare$
\end{note}
\medskip

\begin{note}[Rules of Inference]
{\qzerou}'s R1 rule of inference, Quasi-Equality Substitution,
corresponds to {\qzero}'s single rule of inference, which is equality
substitution.  These rules are exactly the same except that the
{\qzerou} rule requires only \emph{quasi-equality} ($\QuasiEqual$)
between the target wff and the substitution wff, while the {\qzero}
rule requires \emph{equality} (=).

{\qzerou}'s R2 rule of inference, Modus Ponens, is a primitive rule of
inference, but modus ponens is a derived rule of inference in
{\qzero}.  Modus ponens must be primitive in {\qzerou} since it is
needed to discharge the definedness conditions on instances of A4, the
schema for beta-reduction, and A11.\hfill$\blacksquare$
\end{note}
\medskip

A \emph{proof} of a $\wff{o}$ $\textbf{A}_o$ in {\qzerou} is a finite
sequence of $\wffs{o}$, ending with $\textbf{A}_o$, such that each
member in the sequence is an instance of an axiom schema of {\qzerou}
or is inferred from preceding members in the sequence by a rule of
inference of {\qzerou}.  A \emph{theorem} of {\qzerou} is a $\wff{o}$
for which there is a proof in {\qzerou}.

Let $\sH$ be a set of $\wffs{o}$.  A \emph{proof of $\textbf{A}_o$
from $\sH$} in {\qzerou} consists of two finite sequences $\sS_1$ and
$\sS_2$ of $\wffs{o}$ such that $\sS_1$ is a proof in {\qzerou},
$\textbf{A}_o$ is the last member of $\sS_2$, and each member
$\textbf{D}_o$ of $\sS_2$ satisfies at least one of the following
conditions:
\be
  
  \item $\textbf{D}_o \in \sH$.

  \item $\textbf{D}_o$ is a member of $\sS_1$ (and hence a theorem of
  {\qzerou}).

  \item $\textbf{D}_o$ is inferred from two preceding members
  $\textbf{A}_\alpha \QuasiEqual \textbf{B}_\alpha$ and $\textbf{C}_o$
  of $\sS_2$ by R1, provided that the occurrence of
  $\textbf{A}_\alpha$ in $\textbf{C}_o$ is not in a well-formed part
  $\lambda \textbf{x}_\beta \textbf{E}_\gamma$ of $\textbf{C}_o$ where
  $\textbf{x}_\beta$ is free in a member of $\sH$ and free in
  $\textbf{A}_\alpha \QuasiEqual \textbf{B}_\beta$.

  \item $\textbf{D}_o$ is inferred from two preceding members of
  $\sS_2$ by R2.

\ee
We write $\proves{\sH}{\textbf{A}_o}$ to mean there is a proof of
$\textbf{A}_o$ from $\sH$ in {\qzerou}.  $\proves{}{\textbf{A}_o}$
is written instead of $\proves{\emptyset}{\textbf{A}_o}$.  Clearly,
$\textbf{A}_o$ is a theorem of {\qzerou} \,{iff}\,
{$\proves{}{\textbf{A}_o}$}.

The next two theorems follow immediately from the definition above.

\medskip

\begin{theorem}[${\rm R1}'$]
If $\proves{\sH}{\textbf{A}_\alpha \QuasiEqual \textbf{B}_\alpha}$ and
$\proves{\sH}{\textbf{C}_o}$, then $\proves{\sH}{\textbf{D}_o}$, where
$\textbf{D}_o$ is the result of replacing one occurrence of
$\textbf{A}_\alpha$ in $\textbf{C}_o$ by an occurrence of
$\textbf{B}_\alpha$, provided that the occurrence of
$\textbf{A}_\alpha$ in $\textbf{C}_o$ is not immediately preceded by
$\lambda$ or in a well-formed part $\lambda \textbf{x}_\beta
\textbf{E}_\gamma$ of $\textbf{C}_o$ where $\textbf{x}_\beta$ is free
in a member of $\sH$ and free in $\textbf{A}_\alpha \QuasiEqual
\textbf{B}_\alpha$.
\end{theorem}

\medskip

\begin{theorem}[${\rm R2}'$]
If $\proves{\sH}{\textbf{A}_o}$ and
$\proves{\sH}{\textbf{A}_o \Implies \textbf{B}_o}$, then
$\proves{\sH}{\textbf{B}_o}$.
\end{theorem}

\section{Some Metatheorems} \label{sec:metatheorems}

In this section we prove some metatheorems of {\qzerou} that are
needed to prove the soundness and completeness of the proof system of
{\qzerou}.

\medskip

\begin{proposition}[Wffs of type $o$ are defined]\label{prop:odefined}
$\proves{}{{\textbf{A}_o\IsDefApp}}$ for all wffs $\textbf{A}_o$.
\end{proposition}

\begin{proof}
Directly implied by axiom schemas A5, A6, and A8.
\end{proof}

\begin{theorem}[Beta-Reduction Rule]
If $\proves{\sH}{{\textbf{A}_\alpha\IsDefApp}}$ and
$\proves{\sH}{\textbf{C}_o}$, then $\proves{\sH}{\textbf{D}_o}$, where
$\textbf{D}_o$ is the result of replacing one occurrence of $[\lambda
\textbf{x}_\alpha \textbf{B}_\beta]\textbf{A}_\alpha$ in
$\textbf{C}_o$ by an occurrence of $\mname{S}^{{\bf x}_\alpha}_{{\bf
A}_\alpha} \textbf{B}_\beta$, provided $\textbf{A}_\alpha$ is free for
$\textbf{x}_\alpha$ in $\textbf{B}_\beta$ and the occurrence of
$[\lambda \textbf{x}_\alpha \textbf{B}_\beta]\textbf{A}_\alpha$ is not
in a well-formed part $\lambda \textbf{y}_\gamma \textbf{E}_\delta$ of
$\textbf{C}_o$ where $\textbf{y}_\gamma$ is free in a member of $\sH$
and free in $[\lambda \textbf{x}_\alpha
\textbf{B}_\beta]\textbf{A}_\alpha$.
\end{theorem}

\begin{proof}
Follows immediately from A4, R$1'$, and R$2'$.
\end{proof}

\begin{lemma}\label{lem:1}
$\proves{}{\textbf{A}_\alpha \QuasiEqual \textbf{A}_\alpha}$.
\end{lemma}

\begin{proof}
Let $\textbf{x}_\alpha$ be a variable that does not occur in
$\textbf{A}_\alpha$.  Then ${\textbf{x}_\alpha\IsDefApp}$ is an
instance of A5, and ${\textbf{x}_\alpha\IsDefApp} \Implies [[\lambda
\textbf{x}_\alpha \textbf{A}_\alpha]\textbf{x}_\alpha \QuasiEqual
\textbf{A}_\alpha]$ is an instance of A4.  By applying R$2'$ to these
two wffs we obtain $\proves{}{[[\lambda \textbf{x}_\alpha
\textbf{A}_\alpha]\textbf{x}_\alpha \QuasiEqual \textbf{A}_\alpha]}$.
The conclusion of the lemma then follows by the Beta-Reduction Rule.
\end{proof}

\begin{lemma}\label{lem:2}
If $\proves{\sH}{{\textbf{A}_\alpha\IsDefApp}}$ and
$\proves{\sH}{{\textbf{B}_\alpha\IsDefApp}}$, then
$\proves{\sH}{\textbf{A}_\alpha \QuasiEqual \textbf{B}_\alpha}$
\,iff\, $\proves{\sH}{\textbf{A}_\alpha = \textbf{B}_\alpha}$.
\end{lemma}

\begin{proof}

($\Rightarrow$): Follows immediately from A11, R$1'$, and R$2'$.

($\Leftarrow$): $\proves{\sH}{[\textbf{A}_\alpha \QuasiEqual
\textbf{B}_\alpha] \QuasiEqual [\textbf{A}_\alpha =
\textbf{B}_\alpha]}$ by the first two hypotheses, A11, and R$2'$.
$\proves{}{[\textbf{A}_\alpha \QuasiEqual \textbf{B}_\alpha]
\QuasiEqual [\textbf{A}_\alpha \QuasiEqual \textbf{B}_\alpha]}$ by
Lemma~\ref{lem:1}.  We obtain $\proves{\sH}{[\textbf{A}_\alpha =
\textbf{B}_\alpha] \QuasiEqual [\textbf{A}_\alpha \QuasiEqual
\textbf{B}_\alpha]}$ by applying R$1'$ to these two statements.  The
conclusion of the lemma then follows by applying R$1'$ to this last
statement and $\proves{\sH}{\textbf{A}_\alpha = \textbf{B}_\alpha}$.
\end{proof}

\begin{corollary}\label{cor:5200}
If $\proves{}{{\textbf{A}_\alpha\IsDefApp}}$, then
$\proves{}{\textbf{A}_\alpha = \textbf{A}_\alpha}$.
\end{corollary}

\begin{proof}
By Lemmas~\ref{lem:1} and~\ref{lem:2}.
\end{proof}

\begin{lemma}\label{lem:3}
If $\proves{}{{\textbf{A}_\alpha\IsDefApp}}$ and
$\proves{}{\textbf{B}_\beta \QuasiEqual \textbf{C}_\beta}$, then
$\proves{}\mname{S}^{{\bf x}_\alpha}_{{\bf A}_\alpha}[\textbf{B}_\beta
\QuasiEqual \textbf{C}_\beta]$, provided $\textbf{A}_\alpha$ is free
for $\textbf{x}_\alpha$ in $\textbf{B}_\beta \QuasiEqual
\textbf{C}_\beta$.
\end{lemma}

\begin{proof}
Follows from Lemma~\ref{lem:1} and the Beta-Reduction Rule in a way
that is similar to the proof of theorem 5209 in~\cite{Andrews02}.
\end{proof}

\begin{corollary}\label{cor:5209}
If $\proves{}{{\textbf{A}_\alpha\IsDefApp}}$ and $\proves{}{\textbf{B}_o
= \textbf{C}_o}$, then $\proves{}\mname{S}^{{\bf x}_\alpha}_{{\bf
A}_\alpha}[\textbf{B}_o = \textbf{C}_o]$, provided $\textbf{A}_\alpha$
is free for $\textbf{x}_\alpha$ in $\textbf{B}_o = \textbf{C}_o$.
\end{corollary}

\begin{proof}
By Proposition~\ref{prop:odefined}, Lemma~\ref{lem:2}, and
Lemma~\ref{lem:3}.
\end{proof}

\begin{lemma}\label{lem:5210}
If $\proves{}{{\textbf{B}_\beta\IsDefApp}}$, then $\proves{}{T_o =
[\textbf{B}_\beta = \textbf{B}_\beta]}$.
\end{lemma}

\begin{proof}
The proof of $\proves{}{T_o = [\textbf{B}_\beta \QuasiEqual
\textbf{B}_\beta]}$ is similar to the proof of theorem 5210
in~\cite{Andrews02} with Corollaries~\ref{cor:5200} and~\ref{cor:5209}
used in place of theorems 5200 and 5209, respectively.  The lemma then
follows from A11, R1, and R2.
\end{proof}

\begin{lemma}\label{lem:5213}
If $\proves{}{\textbf{A}_o = \textbf{B}_o}$ and
$\proves{}{\textbf{C}_o = \textbf{D}_o}$, then
$\proves{}{[\textbf{A}_o = \textbf{B}_o] \And [\textbf{C}_o =
\textbf{D}_o]}$.
\end{lemma}

\begin{proof}
Similar to the proof of theorem 5213 in~\cite{Andrews02} with
Lemma~\ref{lem:5210} used in place of theorem 5210.
\end{proof}

The proofs of the next four theorems are similar to the proofs of
theorems 5215, 5220, 5234, and 5240 except that:

\be

  \item Rule R1 and Lemma~\ref{lem:2} are used in place of rule R.

  \item Rule R$1'$ and Lemma~\ref{lem:2} are used in place of rule
  ${\rm R}'$.

  \item The Beta-Reduction Rule is used in place of the
  $\beta$-Contraction rule.

  \item Corollary~\ref{cor:5200} is used in place of theorem 5200.

  \item Axiom schema A4 and Lemma~\ref{lem:2} are used in place of
  theorem 5207.

  \item Corollary~\ref{cor:5209} is used in place of theorem 5209.

  \item Lemma~\ref{lem:5210} is used in place of theorem 5210.

  \item Lemma~\ref{lem:5213} is used in place of theorem 5213.

  \item Rule R$2'$ is used in place of theorem 5224 (MP).

  \item Axiom schemas A5--A8 are used to discharge definedness
  conditions.

\ee

\begin{theorem}[Universal Instantiation]\label{thm:5215}
If $\proves{\sH}{{\textbf{A}_\alpha \IsDefApp}}$ and
$\proves{\sH}{\forall \textbf{x}_\alpha \textbf{B}_o}$, then
$\proves{\sH}{\mname{S}^{{\bf x}_\alpha}_{{\bf A}_\alpha}
\textbf{B}_o}$, provided $\textbf{A}_\alpha$ is free for
$\textbf{x}_\alpha$ in $\textbf{B}_o$.
\end{theorem}

\begin{proof}
Similar to the proof of theorem 5215 ($\forall{\rm I}$)
in~\cite{Andrews02}.  See the comment above.
\end{proof}

\begin{theorem}[Universal Generalization]\label{thm:gen}
If $\proves{\sH}{\textbf{A}_o}$, then $\proves{\sH}{\forall
\textbf{x}_\alpha \textbf{A}_o}$, provided $\textbf{x}_\alpha$ is not
free in any wff in $\sH$.
\end{theorem}

\begin{proof}
Similar to the proof of theorem 5220 (Gen) in~\cite{Andrews02}.
See the comment above.
\end{proof}

\begin{theorem}[Tautology Theorem]\label{thm:tautology}
If $\proves{\sH}{\textbf{A}^{1}_{o}}$, \ldots,
$\proves{\sH}{\textbf{A}^{n}_{o}}$ and $[\textbf{A}^{1}_{o} \And
\cdots \And \textbf{A}^{n}_{o}] \Implies \textbf{B}_o$ is tautologous
for $n \ge 1$, then $\proves{\sH}{\textbf{B}_o}$.  Also, if
$\textbf{B}_o$ is tautologous, then $\proves{\sH}{\textbf{B}_o}$.
\end{theorem}

\begin{proof}
Similar to the proof of theorem 5234 (Rule P) in~\cite{Andrews02}.
See the comment above.
\end{proof}

\begin{proposition}\label{prop:1}
$\proves{}{[\textbf{A}_\alpha = \textbf{B}_\alpha] \Implies
[\textbf{A}_\alpha \QuasiEqual \textbf{B}_\alpha]}.$
\end{proposition}

\begin{proof}
Follows from the definition of $\QuasiEqual$ and the Tautology Theorem.
\end{proof}

\begin{theorem}[Deduction Theorem]\label{thm:deduction}
If $\proves{\sH \cup \set{\textbf{H}_o}}{\textbf{P}_o}$, then
$\proves{\sH}{\textbf{H}_o \Implies \textbf{P}_o}$.
\end{theorem}

\begin{proof}
Similar to the proof of theorem 5240 in~\cite{Andrews02}.  See the
comment above.
\end{proof}

\section{Soundness and Completeness} \label{sec:completeness}

In this section, let $\sH$ be a set of $\wffs{o}$.  $\sH$ is
\emph{consistent} if there is no proof of $F_o$ from $\sH$ in
{\qzerou}.

\medskip

\begin{theorem}[Soundness Theorem]
If $\proves{\sH}{\textbf{A}_o}$, then $\sH \models \textbf{A}_o$.
\end{theorem}

\begin{proof}
A straightforward verification shows that (1) each instance of each
axiom schema of {\qzerou} is valid and (2) the rules of inference of
{\qzerou}, R1 and R2, preserve validity in every general model for
{\qzerou}.  This shows that if $\proves{}{\textbf{A}_o}$, then ${}
\models \textbf{A}_o$.

Suppose $\proves{\sH}{\textbf{A}_o}$ and $\sM$ is a model for {\sH}.
Then there is a finite subset
$\set{\textbf{H}^{1}_{o},\ldots,\textbf{H}^{n}_{o}}$ of $\sH$ such
that
$\proves{\set{\textbf{H}^{1}_{o},\ldots,\textbf{H}^{n}_{o}}}{\textbf{A}_o}$.
By the Deduction Theorem, this implies $\proves{}{\textbf{H}^{1}_{o}
\Implies \cdots \Implies \textbf{H}^{n}_{o} \Implies \textbf{A}_o}$.
By the result just above, $\sM \models {\textbf{H}^{1}_{o} \Implies
\cdots \Implies \textbf{H}^{n}_{o} \Implies \textbf{A}_o}$.  But $\sM
\models \textbf{H}^{i}_{o}$ for all $i$ with $1 \le i \le n$ since
$\sM$ is a model for {\sH}.  Therefore $\sM \models \textbf{A}_o$, and
so $\sH \models \textbf{A}_o$.
\end{proof}

\begin{theorem}[Consistency Theorem]
If $\sH$ has a general model, then $\sH$ is consistent.
\end{theorem}

\begin{proof}
Let $\sM$ be a general model for $\sH$.  Assume that $\sH$ is
inconsistent, i.e., that $\proves{\sH}{F_o}$.  Then, by the Soundness
Theorem, $\sH \models F_o$ and hence $\sM \models F_o$.  This means
that $\sV^{\cal M}_{\phi}(F_o) = \mname{T}$ (for any assignment
$\phi$), which contradicts the definition of a general model.
\end{proof}

A \emph{cwff} [$\textit{cwff}_\alpha$] is a closed wff [closed
$\wff{\alpha}$].  A \emph{sentence} is a $\cwff{o}$.  Let $\sH$ be
a set of sentences.  $\sH$ is \emph{complete} in {\qzerou} if, for
every sentence $\textbf{A}_o$, either $\proves{\sH}{\textbf{A}_o}$ or
$\proves{\sH}{\NegAlt\textbf{A}_o}$.  $\sH$ is \emph{extensionally
complete} in {\qzerou} if, for every sentence of the form
$\textbf{A}_{\alpha\beta} = \textbf{B}_{\alpha\beta}$, there is a cwff
$\textbf{C}_\beta$ such that:

\be

  \item $\proves{\sH}{{\textbf{C}_\beta\IsDefApp}}$.

  \item $\proves{\sH}{[{\textbf{A}_{\alpha\beta}\IsDefApp} \And
  {\textbf{B}_{\alpha\beta}\IsDefApp} \And [\textbf{A}_{\alpha\beta}
  \textbf{C}_\beta \QuasiEqual \textbf{B}_{\alpha\beta}
  \textbf{C}_\beta]] \Implies [\textbf{A}_{\alpha\beta} =
  \textbf{B}_{\alpha\beta}]}$.

\ee
Let $\sL(\mbox{\qzerou})$ be the set of wffs of {\qzerou}.

\medskip

\begin{lemma}[Extension Lemma]
Let {\sG} be a consistent set of sentences of {\qzerou}.  Then there
is an expansion {\qzerouplus} of {\qzerou} and a set $\sH$ of
sentences of {\qzerouplus} such that:

\be

  \item $\sG \subseteq \sH$.

  \item $\sH$ is consistent.

  \item $\sH$ is complete in {\qzerouplus}.

  \item $\sH$ is extensionally complete in {\qzerouplus}.

  \item $\mbox{card}(\sL(\mbox{\qzerouplus})) =
  \mbox{card}(\sL(\mbox{\qzerou}))$.

\ee
\end{lemma}

\begin{proof}
The proof is very close to the proof of theorem 5500
in~\cite{Andrews02}.  The crucial difference is that, in case (c) of
the definition of $\sG_{\tau + 1}$,
\[\sG_{\tau + 1} = \sG_\tau \cup
\set{\NegAlt[{\textbf{A}_{\alpha\beta}\IsDefApp} \And
{\textbf{B}_{\alpha\beta}\IsDefApp} \And [\textbf{A}_{\alpha\beta}
\textbf{c}_\beta \QuasiEqual \textbf{B}_{\alpha\beta}
\textbf{c}_\beta]]}\] where $\textbf{c}_\beta$ is the first constant
in $\sC_\beta$ that does not occur in $\sG_\tau$ or
$\textbf{A}_{\alpha\beta} = \textbf{B}_{\alpha\beta}$.  (Notice that
$\proves{}{\textbf{c}_\beta\IsDefApp}$ by A6.)

To prove that $\sG_{\tau + 1}$ is consistent assuming $\sG_\tau$ is
consistent when $\sG_{\tau + 1}$ is obtained by case (c) , it is
necessary to show that, if
\[\proves{\sG_\tau}{{\textbf{A}_{\alpha\beta}\IsDefApp} \And
{\textbf{B}_{\alpha\beta}\IsDefApp} \And [\textbf{A}_{\alpha\beta}
\textbf{c}_\beta \QuasiEqual \textbf{B}_{\alpha\beta}
\textbf{c}_\beta]},\] then $\proves{\sG_\tau}{\textbf{A}_{\alpha\beta}
= \textbf{B}_{\alpha\beta}}$.  Assume the hypothesis of this
statement.  Let $\sP$ be a proof of
\[{\textbf{A}_{\alpha\beta}\IsDefApp} \And 
{\textbf{B}_{\alpha\beta}\IsDefApp} \And [\textbf{A}_{\alpha\beta}
  \textbf{c}_\beta \QuasiEqual \textbf{B}_{\alpha\beta}
  \textbf{c}_\beta]\] from a finite subset $\sS$ of $\sG_\tau$, and
let $\textbf{x}_\beta$ be a variable that does not occur in $\sP$ or
$\sS$.  Since $\textbf{c}_\beta$ does not occur in $\sG_\tau$,
$\textbf{A}_{\alpha\beta}$, or $\textbf{B}_{\alpha\beta}$, the result
of substituting $\textbf{x}_\beta$ for each occurrence of
$\textbf{c}_\beta$ in $\sP$ is a proof of
\[{\textbf{A}_{\alpha\beta}\IsDefApp} \And
{\textbf{B}_{\alpha\beta}\IsDefApp} \And [\textbf{A}_{\alpha\beta}
\textbf{x}_\beta \QuasiEqual \textbf{B}_{\alpha\beta}
\textbf{x}_\beta]\] from $\sS$.  Therefore,
\[\proves{\sS}{{\textbf{A}_{\alpha\beta}\IsDefApp} \And
{\textbf{B}_{\alpha\beta}\IsDefApp} \And [\textbf{A}_{\alpha\beta}
\textbf{x}_\beta \QuasiEqual \textbf{B}_{\alpha\beta}
\textbf{x}_\beta]}.\] This implies
\[\proves{\sS}{{\textbf{A}_{\alpha\beta}\IsDefApp}}, \sglsp
\proves{\sS}{{\textbf{B}_{\alpha\beta}\IsDefApp}}, \sglsp
\proves{\sS}{\forall \textbf{x}_\beta [\textbf{A}_{\alpha\beta}
    \textbf{x}_\beta \QuasiEqual \textbf{B}_{\alpha\beta}
    \textbf{x}_\beta]}\] by the Tautology Theorem and Universal
Generalization since $\textbf{x}_\beta$ does not occur in $\sS$.  It
follows from these that $\proves{\sG_\tau}{\textbf{A}_{\alpha\beta} =
  \textbf{B}_{\alpha\beta}}$ by A3, Lemma~\ref{lem:1},
Proposition~\ref{prop:1}, Universal Generalization, Universal
Instantiation, R$1'$, and R$2'$.

The rest of the proof is essentially the same as the proof of theorem
5500.
\end{proof}

\bsp A general model $\seq{\set{\sD_\alpha \;|\; \alpha \in \sT}, \sJ}$ for
{\qzerou} is \emph{frugal} if $\mbox{card}(\sD_\alpha) \le
\mbox{card}(\sL(\mbox{\qzerou}))$ for all $\alpha \in \sT$. \esp

\medskip

\begin{theorem}[Henkin's Theorem for {\qzerou}]
Every consistent set of sentences of {\qzerou} has a frugal general
model.
\end{theorem}

\begin{proof} 
Let $\sG$ be a consistent set of sentences of {\qzerou}, and let $\sH$
and {\qzerouplus} be as described in the Extension Lemma.  We define
simultaneously, by induction on $\gamma \in \sT$, a frame
$\set{\sD_\alpha \;|\; \alpha \in \sT}$ and a partial function $\sV$
whose domain is the set of cwffs of {\qzerouplus} so that the
following conditions hold for all $\gamma \in \sT$:

\bi

  \item [$(1^\gamma)$] $\sD_\gamma = \set{\sV(\textbf{A}_\gamma) \;|\;
  \textbf{A}_\gamma \mbox{ is a } \cwff{\gamma} \mbox{ and }
  \proves{\sH}{{\textbf{A}_\gamma\IsDefApp}}}$.

  \item [$(2^\gamma)$] $\sV(\textbf{A}_\gamma)$ is defined \,iff\,
  $\proves{\sH}{{\textbf{A}_\gamma\IsDefApp}}$ for all cwffs
  $\textbf{A}_\gamma$.

  \item [$(3^\gamma)$] $\sV(\textbf{A}_\gamma) =
  \sV(\textbf{B}_\gamma)$ \,iff\, $\proves{\sH}{\textbf{A}_\gamma =
  \textbf{B}_\gamma}$ for all cwffs $\textbf{A}_\gamma$ and
  $\textbf{B}_\gamma$.

\ei
Let $\sV(x) \QuasiEqual \sV(y)$ mean either $\sV(x)$ and $\sV(y)$ are
both defined and equal or $\sV(x)$ and $\sV(y)$ are both undefined.

For each cwff $\textbf{A}_o$, if $\proves{\sH}{\textbf{A}_o}$, let
$\sV(\textbf{A}_o) = \mname{T}$, and otherwise let $\sV(\textbf{A}_o)
= \mname{F}$.  Also, let $\sD_o = \set{\mname{T},\mname{F}}$.  By the
consistency and completeness of $\sH$, exactly one of
$\proves{\sH}{\textbf{A}_o}$ and $\proves{\sH}{\NegAlt\textbf{A}_o}$
holds.  Hence $(1^o)$ and $(3^o)$ are satisfied.  $(2^o)$ is satisfied
by Proposition~\ref{prop:odefined}.

For each cwff $\textbf{A}_{\iotaAltS}$, if
$\proves{\sH}{{\textbf{A}_{\iotaAltS}\IsDefApp}}$, let
\[\sV(\textbf{A}_{\iotaAltS}) = \set{\textbf{B}_{\iotaAltS} \;|\;
\textbf{B}_{\iotaAltS} \mbox{ is a } \cwff{\iotaAltS} \mbox{ and }
\proves{\sH}{\textbf{A}_{\iotaAltS} = \textbf{B}_{\iotaAltS}}},\] and
otherwise let $\sV(\textbf{A}_{\iotaAltS})$ be undefined.  Also, let
\[\sD_{\iotaAltS} = \set{\sV(\textbf{A}_{\iotaAltS}) \;|\; 
\textbf{A}_{\iotaAltS} \mbox{ is a } \cwff{\iotaAltS} \mbox{ and }
\proves{\sH}{{\textbf{A}_{\iotaAltS}\IsDefApp}}}.\] $(1^{\iotaAltS})$,
$(2^{\iotaAltS})$, and $(3^{\iotaAltS})$ are clearly satisfied.

Now suppose that $\sD_\alpha$ and $\sD_{\beta}$ are defined and that
the conditions hold for $\alpha$ and $\beta$.  For each cwff
$\textbf{A}_{\alpha\beta}$, if
$\proves{\sH}{{\textbf{A}_{\alpha\beta}\IsDefApp}}$, let
$\sV(\textbf{A}_{\alpha\beta})$ be the (partial or total) function
from $\sD_{\beta}$ to $\sD_\alpha$ whose value, for any argument
$\sV(\textbf{B}_\beta) \in \sD_\beta$, is
$\sV(\textbf{A}_{\alpha\beta}\textbf{B}_\beta)$ if
$\sV(\textbf{A}_{\alpha\beta}\textbf{B}_\beta)$ is defined and is
undefined if $\sV(\textbf{A}_{\alpha\beta}\textbf{B}_\beta)$ is
undefined, and otherwise let $\sV(\textbf{A}_{\alpha\beta})$ be
undefined.  We must show that this definition is independent of the
particular cwff $\textbf{B}_\beta$ used to represent the argument.  So
suppose $\sV(\textbf{B}_\beta) = \sV(\textbf{C}_\beta)$; then
$\proves{\sH}{\textbf{B}_\beta = \textbf{C}_\beta}$ by $(3^{\beta})$,
so $\proves{\sH}{\textbf{A}_{\alpha\beta}\textbf{B}_\beta \QuasiEqual
  \textbf{A}_{\alpha\beta}\textbf{C}_\beta}$ by Lemmas~\ref{lem:1}
and~\ref{lem:2} and R$1'$, and so
$\sV(\textbf{A}_{\alpha\beta}\textbf{B}_\beta) \QuasiEqual
\sV(\textbf{A}_{\alpha\beta}\textbf{C}_\beta)$ by $(2^{\alpha})$ and
$(3^{\alpha})$, Finally, let \[\sD_{\alpha\beta} =
\set{\sV(\textbf{A}_{\alpha\beta}) \;|\; \textbf{A}_{\alpha\beta}
  \mbox{ is a } \cwff{\alpha\beta} \mbox{ and }
  \proves{\sH}{{\textbf{A}_{\alpha\beta}\IsDefApp}}}.\]
$(1^{\alpha\beta})$ and $(2^{\alpha\beta})$ are clearly satisfied; we
must show that $(3^{\alpha\beta})$ is satisfied.  Suppose
$\sV(\textbf{A}_{\alpha\beta}) = \sV(\textbf{B}_{\alpha\beta})$.  Then
$\proves{\sH}{{\textbf{A}_{\alpha\beta}\IsDefApp}}$ and
$\proves{\sH}{{\textbf{B}_{\alpha\beta}\IsDefApp}}$.  Since $\sH$ is
extensionally complete, there is a $\textbf{C}_\beta$ such that
$\proves{\sH}{{\textbf{C}_{\beta}\IsDefApp}}$ and
\[\proves{\sH}{[{\textbf{A}_{\alpha\beta}\IsDefApp} \And
{\textbf{B}_{\alpha\beta}\IsDefApp} \And [\textbf{A}_{\alpha\beta}
  \textbf{C}_\beta \QuasiEqual \textbf{B}_{\alpha\beta}
  \textbf{C}_\beta]] \Implies [\textbf{A}_{\alpha\beta} =
    \textbf{B}_{\alpha\beta}]}.\] Then
$\sV(\textbf{A}_{\alpha\beta}\textbf{C}_\beta) \QuasiEqual
\sV(\textbf{A}_{\alpha\beta})(\sV(\textbf{C}_\beta)) \QuasiEqual
\sV(\textbf{B}_{\alpha\beta})(\sV(\textbf{C}_\beta)) \QuasiEqual
\sV(\textbf{B}_{\alpha\beta}\textbf{C}_\beta),$ so
$\proves{\sH}{\textbf{A}_{\alpha\beta}\textbf{C}_\beta \QuasiEqual
  \textbf{B}_{\alpha\beta}\textbf{C}_\beta}$ by $(2^\alpha)$ and
$(3^\alpha)$, and so $\proves{\sH}{\textbf{A}_{\alpha\beta} =
  \textbf{B}_{\alpha\beta}}$.  Now suppose
$\proves{\sH}{\textbf{A}_{\alpha\beta} = \textbf{B}_{\alpha\beta}}$.
Then, for all cwffs $\textbf{C}_\beta \in \sD_\beta$,
$\proves{\sH}{\textbf{A}_{\alpha\beta}\textbf{C}_\beta \QuasiEqual
  \textbf{B}_{\alpha\beta}\textbf{C}_\beta}$ by Lemmas~\ref{lem:1}
and~\ref{lem:2} and R$1'$, and so
$\sV(\textbf{A}_{\alpha\beta})(\sV(\textbf{C}_\beta)) \QuasiEqual
\sV(\textbf{A}_{\alpha\beta}\textbf{C}_\beta) \QuasiEqual
\sV(\textbf{B}_{\alpha\beta}\textbf{C}_\beta) \QuasiEqual
\sV(\textbf{B}_{\alpha\beta})(\sV(\textbf{C}_\beta)).$ Hence
$\sV(\textbf{A}_{\alpha\beta}) = \sV(\textbf{B}_{\alpha\beta})$.

We claim that $\sM = \seq{\set{\sD_\alpha \;|\; \alpha \in \sT}, \sV}$
is an interpretation.  For each primitive constant $\textbf{c}_\gamma$
of {\qzerouplus}, $\proves{\sH}{\textbf{c}_\gamma}$ by A6, and thus
$\sV$ maps each primitive constant of {\qzerouplus} of type $\gamma$
into $\sD_\gamma$ by $(1^{\gamma})$ and $(2^{\gamma})$.

\bsp We must show that $\sV(\mname{Q}_{o\alpha\alpha})$ is the
identity relation on $\sD_\alpha$.  Let $\sV(\textbf{A}_\alpha)$ and
$\sV(\textbf{B}_\alpha)$ be arbitrary members of $\sD_\alpha$.  Then
$\sV(\textbf{A}_\alpha) = \sV(\textbf{B}_\alpha)$ iff
$\proves{\sH}{\textbf{A}_\alpha = \textbf{B}_\alpha}$ iff
$\proves{\sH}{\mname{Q}_{o\alpha\alpha}\textbf{A}_\alpha\textbf{B}_\alpha}$
iff $\mname{T} =
\sV(\mname{Q}_{o\alpha\alpha}\textbf{A}_\alpha\textbf{B}_\alpha) =
\sV(\mname{Q}_{o\alpha\alpha})
(\sV(\textbf{A}_\alpha))(\sV(\textbf{B}_\alpha)).$ Thus
$\sV(\mname{Q}_{o\alpha\alpha})$ is the identity relation on
$\sD_\alpha$. \esp

We must show that, for $\alpha \not= o$,
$\sV(\iota_{\alpha(o\alpha)})$ is the unique member selector on
$\sD_\alpha$.  For $\alpha \not= o$, let $\textbf{A}_{o\alpha}$ be an
arbitrary member of $\sD_{o\alpha}$, $\textbf{B}_\alpha$ be an
arbitrary member of $\sD_\alpha$, and $\textbf{x}_\alpha$ be a
variable that does not occur in $\textbf{A}_{o\alpha}$.  Using
A12 and A13, $\sV(\textbf{A}_{o\alpha}) =
\sV(\mname{Q}_{o\alpha\alpha} \textbf{B}_\alpha)$ iff
$\proves{\sH}{\textbf{A}_{o\alpha} = \mname{Q}_{o\alpha\alpha}
\textbf{B}_\alpha}$ iff
$\proves{\sH}{\iota_{\alpha(o\alpha)}\textbf{A}_{o\alpha} =
\textbf{B}_\alpha}$ iff
$\sV(\iota_{\alpha(o\alpha)}\textbf{A}_{o\alpha}) = \sV
(\textbf{B}_\alpha)$ iff
$\sV(\iota_{\alpha(o\alpha)})(\sV(\textbf{A}_{o\alpha})) = \sV
(\textbf{B}_\alpha)$.  Similarly, using A12 and A13,
$\sV(\NegAlt\Forsome_1 \textbf{x}_\alpha
[\textbf{A}_{o\alpha}\textbf{x}_\alpha]) = \mname{T}$ iff
$\proves{\sH}{\NegAlt\Forsome_1 \textbf{x}_\alpha
[\textbf{A}_{o\alpha}\textbf{x}_\alpha]}$ iff
$\proves{\sH}{\iota_{\alpha(o\alpha)}\textbf{A}_{o\alpha}
\IsUndefApp}$ iff $\sV(\iota_{\alpha(o\alpha)}\textbf{A}_{o\alpha})$
is undefined iff $\sV(\iota_{\alpha(o\alpha)})
(\sV(\textbf{A}_{o\alpha}))$ is undefined.  Thus
$\sV(\iota_{\alpha(o\alpha)})$ is the unique member selector on
$\sD_\alpha$.

Thus $\sM$ is an interpretation.  We claim further that $\sM$ is a
general model for {\qzerouplus}.  For each assignment $\phi$ into
$\sM$ and wff $\textbf{C}_\gamma$, let \[\textbf{C}^{\phi}_{\gamma} =
\mname{S}^{{\bf x}^{1}_{\delta_1} \cdots {\bf x}^{n}_{\delta_n}}_{{\bf
    E}^{1}_{\delta_1} \cdots {\bf E}^{n}_{\delta_n}}
\textbf{C}_\gamma\] where ${\bf x}^{1}_{\delta_1} \cdots {\bf
  x}^{n}_{\delta_n}$ are the free variables of $\textbf{C}_\gamma$ and
${\bf E}^{i}_{\delta_i}$ is the first cwff (in some fixed enumeration)
of {\qzerouplus} such that $\phi({\bf x}^{i}_{\delta_i}) = \sV({\bf
  E}^{i}_{\delta_i})$ for all $i$ with $1 \le i \le
n$.\footnote{$\mname{S}^{{\bf x}^{1}_{\delta_1} \cdots {\bf
      x}^{n}_{\delta_n}}_{{\bf E}^{1}_{\delta_1} \cdots {\bf
      E}^{n}_{\delta_n}} \textbf{C}_\gamma$ is the result of
  simultaneously substituting $\textbf{E}^{i}_{\delta_i}$ for each
  free occurrence of $\textbf{x}^{i}_{\delta_i}$ in
  $\textbf{C}_\gamma$ for all $i$ with $1 \le i \le n$.}  Let
$\sV_\phi(\textbf{C}_\gamma) \QuasiEqual
\sV(\textbf{C}^{\phi}_{\gamma})$.  $\textbf{C}^{\phi}_{\gamma}$ is
clearly a $\cwff{\gamma}$, so $\sV_\phi(\textbf{C}_\gamma) \in
\sD_\gamma$ if $\sV_\phi(\textbf{C}_\gamma)$ is defined.

\be

  \item[(a)] Let $\textbf{C}_\gamma$ be a variable
  $\textbf{x}_\delta$.  Choose $\textbf{E}_\delta$ so that
  $\phi(\textbf{x}_\delta) = \sV(\textbf{E}_\delta)$ as above.  Then
  $\sV_\phi(\textbf{C}_\gamma) = \sV_\phi(\textbf{x}_\delta) =
  \sV(\textbf{x}^{\phi}_{\delta}) = \sV(\textbf{E}_\delta) =
  \phi(\textbf{x}_\delta)$.

  \item[(b)] Let $\textbf{C}_\gamma$ be a primitive constant.  Then
    $\sV_\phi(\textbf{C}_\gamma) = \sV(\textbf{C}^{\phi}_{\gamma}) =
    \sV(\textbf{C}_{\gamma})$.

  \item[(c)] \bsp Let $\textbf{C}_\gamma$ be
  $[\textbf{A}_{\alpha\beta} \textbf{B}_\beta]$.  If
  $\sV_{\phi}(\textbf{A}_{\alpha\beta})$ is defined,
  $\sV_{\phi}(\textbf{B}_\beta)$ is defined, and
  $\sV_{\phi}(\textbf{A}_{\alpha\beta})$ is defined at
  $\sV_{\phi}(\textbf{B}_\beta)$, then $\sV_{\phi}(\textbf{C}_\gamma)
  = \sV_{\phi}(\textbf{A}_{\alpha\beta}\textbf{B}_\beta) =
  \sV(\textbf{A}^{\phi}_{\alpha\beta}\textbf{B}^{\phi}_\beta) =
  \sV(\textbf{A}^{\phi}_{\alpha\beta})(\sV(\textbf{B}^{\phi}_\beta)) =
  \sV_{\phi}(\textbf{A}_{\alpha\beta})(\sV_{\phi}(\textbf{B}_\beta)).$
  Now assume $\sV_{\phi}(\textbf{A}_{\alpha\beta})$ is undefined,
  $\sV_{\phi}(\textbf{B}_\beta)$ is undefined, or
  $\sV_{\phi}(\textbf{A}_{\alpha\beta})$ is not defined at
  $\sV_{\phi}(\textbf{B}_\beta)$.  Then
  $\proves{\sH}{{\textbf{A}^{\phi}_{\alpha\beta}\IsUndefApp}}$,
  $\proves{\sH}{{\textbf{B}^{\phi}_{\beta}\IsUndefApp}}$, or
  $\sV(\textbf{A}^{\phi}_{\alpha\beta}\textbf{B}^{\phi}_{\beta})$ is
  undefined.  If $\alpha = o$, then
  $\proves{\sH}{{\textbf{A}^{\phi}_{\alpha\beta}\IsUndefApp}}$ or
  $\proves{\sH}{{\textbf{B}^{\phi}_{\beta}\IsUndefApp}}$, which implies
  $\proves{\sH}{\NegAlt
  \textbf{A}^{\phi}_{\alpha\beta}\textbf{B}^{\phi}_{\beta}}$ by A9, so
  $\proves{\sH}{\textbf{A}^{\phi}_{\alpha\beta}\textbf{B}^{\phi}_{\beta} =
  F_o}$, so $\sV(\textbf{A}^{\phi}_{\alpha\beta}\textbf{B}^{\phi}_{\beta})
  = \sV(F_o)$, so $\sV_\phi(\textbf{A}_{\alpha\beta}\textbf{B}_\beta) =
  \sV(F_o)$, and so $\sV_\phi(\textbf{C}_\gamma) = \mname{F}$.  If
  $\alpha \not= o$ and
  $\proves{\sH}{{\textbf{A}^{\phi}_{\alpha\beta}\IsUndefApp}}$ or
  $\proves{\sH}{{\textbf{B}^{\phi}_{\beta}\IsUndefApp}}$, then
  $\proves{\sH}{{[\textbf{A}^{\phi}_{\alpha\beta}
  \textbf{B}^{\phi}_{\beta}]\IsUndefApp}}$ by A10, and so
  $\sV(\textbf{A}^{\phi}_{\alpha\beta}\textbf{B}^{\phi}_{\beta})$ is
  undefined.  Hence, if $\alpha \not= o$,
  $\sV_{\phi}(\textbf{C}_\gamma) \QuasiEqual
  \sV_{\phi}(\textbf{A}_{\alpha\beta}\textbf{B}_\beta) \QuasiEqual
  \sV(\textbf{A}^{\phi}_{\alpha\beta}\textbf{B}^{\phi}_\beta)$ is
  undefined. \esp

  \item[(d)] Let $\textbf{C}_\gamma$ be
  $[\lambda\textbf{x}_\alpha\textbf{B}_\beta]$.  Let
  $\sV(\textbf{E}_\alpha)$ be an arbitrary member of $\sD_\alpha$, and
  so $\textbf{E}_\alpha$ is a cwff and
  $\proves{\sH}{{\textbf{E}_\alpha\IsDefApp}}$.  Given an assignment
  $\phi$, let $\psi = (\phi :
  \textbf{x}_\alpha/\sV(\textbf{E}_\alpha)$.  From A4 it follows that
  $\proves{\sH}{[\lambda \textbf{x}_\alpha \textbf{B}_\beta]^{\phi}
  \textbf{E}_{\alpha} \QuasiEqual \textbf{B}^{\psi}_{\beta}}$.  Then
  $\sV_{\phi}(\textbf{C}_\gamma)(\sV(\textbf{E}_\alpha)) \QuasiEqual
  \sV([\lambda\textbf{x}_\alpha\textbf{B}_\beta]^{\phi})
  (\sV(\textbf{E}_\alpha)) \QuasiEqual \sV(\textbf{B}^{\psi}_{\beta})
  \QuasiEqual \sV_\psi(\textbf{B}_{\beta}).$ Thus
  $\sV_{\phi}(\textbf{C}_\gamma)$ satisfies condition (d) in the
  definition of a general model.

\ee

Thus $\sM$ is a general model for {\qzerouplus} (and hence for
{\qzerou}).  Also, if $\textbf{A}_o \in \sG$, then $\textbf{A}_o \in
\sH$, so $\proves{\sH}{\textbf{A}_o}$, so $\sV(\textbf{A}_o) =
\mname{T}$ and $\sM \models \textbf{A}_o$, so $\sM$ is a general model
for $\sG$.  Clearly, (1) $\mbox{card}(\sD_\alpha) \le
\mbox{card}(\sL(\mbox{\qzerou}))$ since $\sV$ maps a subset of the
$\cwffs{\alpha}$ of {\qzerouplus} onto $\sD_\alpha$ and (2)
$\mbox{card}(\sL(\mbox{\qzerouplus})) =
\mbox{card}(\sL(\mbox{\qzerou}))$, and so $\sM$ is frugal.
\end{proof}

\begin{theorem}[Henkin's Completeness Theorem for {\qzerou}]
Let $\sH$ be a set of sentences of {\qzerou}.  If $\sH \models
\textbf{A}_o$, then $\proves{\sH}{\textbf{A}_o}$.
\end{theorem}

\begin{proof} 
Assume $\sH \models \textbf{A}_o$, and let $\textbf{B}_o$ be the
universal closure of $\textbf{A}_o$.  Then $\sH \models \textbf{B}_o$.
Suppose $\sH \cup \{\NegAlt \textbf{B}_o\}$ is consistent.  Then, by
Henkin's Theorem, there is a general model $\sM$ for $\sH \cup
\set{\NegAlt \textbf{B}_o}$, and so $\sM \models \NegAlt
\textbf{B}_o$.  Since $\sM$ is also a general model for $\sH$,
$\sM \models \textbf{B}_o$.  From this contradiction it follows that
$\sH \cup \set{\NegAlt \textbf{B}_o}$ is inconsistent.  Hence
$\proves{\sH}{\textbf{B}_o}$ by the Deduction Theorem and the
Tautology Theorem.  Therefore, $\proves{\sH}{\textbf{A}_o}$ by
Universal Instantiation and A5.
\end{proof}

\section{Conclusion} \label{sec:conclusion}

{\qzerou} is a version of Church's type theory that directly
formalizes the traditional approach to undefinedness.  In this paper
we have presented the syntax, semantics, and proof system of
{\qzerou}.  The semantics is based on Henkin-style general models.  We
have also proved that {\qzerou} is sound and complete with respect to
its semantics.

{\qzerou} is a modification of {\qzero}.  Its syntax is essentially
identical to the syntax of {\qzero}.  Its semantics is based on
general models that include partial functions as well as total
functions and in which terms may be nondenoting.  Its proof system is
derived from the proof system of {\qzero}; the axiom schemas and rules
of inference of {\qzero} have been modified to accommodate partial
functions and undefined terms and to axiomatize definite description.

Our presentation of {\qzerou} is intended to show as clearly as
possible what must be changed in Church's type theory in order to
formalize the traditional approach to undefinedness.  Our development
of {\qzerou} closely follows Andrews' development of {\qzero}.  Notes
indicate where and how {\qzero} and {\qzerou} differ from each other.
And the proofs of the soundness and completeness theorems for
{\qzerou} follow very closely the proofs of these theorems for
{\qzero}.

{\qzero} and {\qzerou} have the same \emph{theoretical expressivity}
(see Note~\ref{note:mutual}).  However, with its formalization of the
traditional approach, {\qzerou} has significantly greater
\emph{practical expressivity} than {\qzero}.  Statements involving
partial functions and undefined terms can be expressed in {\qzerou}
more naturally and concisely than in {\qzero} (see~\cite{Farmer04}).
All the standard laws of predicate logic hold in {\qzerou} except
those involving equality and substitution, but these do hold for
defined terms.  In summary, {\qzerou} has the benefit of greater
practical expressivity at the cost of a modest departure from standard
predicate logic.

The benefits of a practical logic like {\qzerou} would be best
realized by a computer implementation of the logic.  {\qzerou} has not
been implemented, but the related logic
{\lutins}~\cite{Farmer90,Farmer93b,Farmer94} has been implemented in
the {\imps} theorem proving system~\cite{FarmerEtAl93,FarmerEtAl96}
and successfully used to prove hundreds of theorems in traditional
mathematics, especially in mathematical analysis.  {\lutins} is
essentially just a more sophisticated version of {\qzerou} with
subtypes and additional expression constructors.  An implemented logic
that formalizes the traditional approach to undefinedness can reap the
benefits of a proven approach developed in mathematical practice over
hundreds of years.

\section{Acknowledgments} 

Peter Andrews deserves special thanks for writing \emph{An
Introduction to Mathematical Logic and Type Theory: To Truth through
Proof}~\cite{Andrews02}.  The ideas embodied in {\qzerou} heavily
depend on the presentation of {\qzero} given in this superb textbook.

\bibliography{$HOME/research/lib/imps}
\bibliographystyle{plain-casefix}

\end{document}